\documentclass[12pt]{article}

\usepackage{amsmath,amssymb,amsfonts,amsthm,graphicx,color,verbatim}

\definecolor{Blue}{rgb}{0.3,0.3,0.9}

\usepackage[pdfstartview=FitH,
            colorlinks,
           linkcolor=reference,
            citecolor=citation,
            urlcolor=e-mail,
]{hyperref}

\RequirePackage{color}
\definecolor{todo}{rgb}{1,0,0}
\definecolor{answer}{rgb}{0,0,1}
\definecolor{new}{rgb}{1,0,1}
\definecolor{conditional}{rgb}{0,1,0}
\definecolor{e-mail}{rgb}{0,.40,.80}
\definecolor{reference}{rgb}{.20,.60,.22}
\definecolor{mrnumber}{rgb}{.80,.40,0}
\definecolor{citation}{rgb}{0,.40,.80}

\newtheorem{theorem}{Theorem}[section]
\newtheorem{lemma}[theorem]{Lemma}

\numberwithin{equation}{section}

\theoremstyle{remark}

\begin{document}

\title{A Mean Field Approximation of the Bolker-Pacala Population Model\protect}
\author{Mariya Bessonov\footnote{The first author was supported in part by NSF grant DMS-0739164.}
 \\\small Department of Mathematics\\[-0.8ex]
\small Cornell University, Ithaca, NY 14853,USA\\
\\
Stanislav Molchanov, Joseph Whitmeyer
 \\\small Department of Mathematics and Statistics\\[-0.8ex]
\small University of North Carolina at Charlotte, Charlotte, NC 28223,USA\\}

\maketitle

\begin{abstract}
We approximate the Bolker-Pacala model of population dynamics with the logistic Markov chain and analyze the latter.   We find the asymptotics of the degenerated hypergeometric function and use these to prove a local CLT and large deviations result.  We also state global limit theorems and obtain asymptotics for the first passage time to the boundary of a large interval.
\end{abstract}

\renewcommand{\thefootnote}{\fnsymbol{footnote}}
\renewcommand{\thefootnote}{\arabic{footnote}}

\bigskip

\section{Introduction}
The central problem in population dynamics is the construction of ``realistic" models with non-trivial ergodic limits,  that is, to construct birth-and-death Markov processes that tend in law to a stationary state.  These models must satisfy, in addition, several other requirements. First, they must possess stability with respect to random noise, which always exists in the environment. Second, they should contain stochastic dynamics (we have developed, for example, models involving migration and, in some cases, immigration \cite{fmw}).  Third, because all bio-populations demonstrate an intermittency effect or ``patches," these models should exhibit deviations from local Poissonian (uniform) statistics. (Note that in this paper by ``model" we simply mean a process with an external phenomenon as a referent, and thus we often use ``model" and ``process" interchangeably.)

The classical Galton-Watson process (see \cite{gs, teh}), for example, cannot satisfy these requirements, as it has two significant defects: there are no spatial dynamics -- indeed, no spatial distribution of particles -- and there is no interaction between particles.

The first defect can be easily remedied if we have the particles undergo independent random walks. Recent work \cite{fmw,kkm11,ks06} contains the important result that if the underlying random walk is transient, then, under mild technical conditions, the point field $n(t,x)$ has a unique stationary ergodic limit $n(\infty,x)$. Unfortunately, situations with other requirements are much more complex.

The fact that the population is stable only in the critical condition that $b = \mu$, for birth rate $b$ and mortality rate $\mu$, is less than satisfactory, given that, in the short run at least, there is no obvious reason why this condition should hold.  Another troublesome fact arises from the work of Kondratiev, Kutoviy, and Molchanov \cite{kkm11}.  They show that for a population in a time-independent, stationary-in-space random environment, there is no non-trivial limiting distribution. Roughly speaking, then, the population is stable only if the parameters specify the critical condition at almost every point in the space. Clearly, it would be preferable to have an alternative model that more robustly yielded a stable distribution--which, after all, is found frequently in nature.  

For this we turn to the Bolker-Pacala model \cite{bp99,bpn03}.  The Bolker-Pacala model, well-known in the theory of population dynamics, is a stochastic spatial model that incorporates both spatial dynamics and competition.
It is of especial concern, however, that we have no good qualitative information about the B-P (or logistic) process. Consequently, in this paper we begin to study the B-P model. Our approach is to consider a discrete version on the lattice $\mathbb{Z}^d$ and then to use a random walk representation in mean field approximation.

The paper is organized as follows. Section~\ref{sec:description} offers a detailed description of the Bolker-Pacala process. The next section describes our mean-field model and presents the logistic Markov chain, which we use to analyze the process.  Section 4 contains our analysis of the logistic Markov chain. These include, after preliminary supporting sections, limit theorems for the invariant distribution, global limit theorems, and analysis of the first passage time to the boundary points of an interval. 

Our principal results are contained in section 4.  We begin by deriving asymptotic properties of the degenerated (confluent) hypergeometric function, which we use to prove a local Central Limit Theorem.  Subsequently, we obtain a large deviations result for the invariant distribution for the logistic Markov chain. Namely, if we make the weight of the quadratic term in the B-P model proportional to $L^{-1}$, then the invariant probability for large deviations from the equilibrium state is dominated by $L^{-1/2}$ times a constant times the exponential of $-L$ (Theorem~\ref{LDT}).  Based on theorems of Kurtz \cite{tk70,tk71}, we state global limit theorems---a functional LLN and CLT---for the logistic Markov chain.  Lastly, we obtain asymptotics for the first passage time from the equilibrium point of the logistic Markov chain.  Specifically, we show that  the first passage time from the equilibrium point to the bounds of a symmetric interval, the left bound of which is close to 0, is dominated by $L^{-1/2}$ 
times a constant times the exponential of $L$ (equation \ref{symmFPT}).


\section{Preliminaries: Description of the Process}\label{sec:description}
In this section, we introduce the general Bolker-Pacala model, which can be formulated as follows. 
At time $t=0$, we have an initial homogeneous population, that is, a locally finite point process $$n_0(\Gamma) = \#\text{(particles in $\Gamma$ at time $t=0$)},$$ where $\Gamma$ denotes a bounded and connected region in $\mathbb{R}^d$.  The simplest option is for $n_0(\Gamma)$ to be a Poissonian point field with intensity $\rho > 0$, i.e., $$P\{n_0(\Gamma) = k\} = \exp(-\rho|\Gamma|)\frac{(\rho|\Gamma|)^k}{k!},\ k=0,1,2,\ldots$$ where $|\Gamma|$ is the finite Lebesgue measure of $\Gamma$.  The following rules dictate the evolution of the field:
\begin{enumerate}
\item[i)] Each particle, independent of the others, during time interval $(t,t+dt)$ can produce a new particle (offspring or seed) with probability $b dt = A^+ dt$, $A^+>0$.  The initial particle remains at its initial position $x$ but the offspring jumps to $x+z+dz$ with probability $$a^+(z) dz,\quad A^+=\int\limits_{\mathbb{R}^d} \!\! a^+(x) dx.$$  Note that this can be seen equivalently as two random events, the birth of a particle and its dispersal, as in Bolker and Pacala's presentation \cite{bp99}, or as a single random event, as in our model.  (We stress that this differs from the classical branching process, in which the ``parental" particle and its offspring commence independent motion from the same point.)  We assume, of course, that all offspring evolve independently according to the same rules. 
\item[ii)] Each particle at point $x$ during the time interval $(t,t+{d}t)$ dies with probability $\mu {d}t$, where  $\mu$ is the mortality rate.
\item[iii)] Most important is the competition factor.  If two particles are located at the points $x,y \in \mathbb{R}^d$, then each of them dies with probability $a^-(x-y)dt$ during the time interval $(t,t+dt)$ (we may assume that both do not die). This requires, of course, that $a^-(\cdot)$ be integrable; set $$A^- = \int\limits_{\mathbb{R}^d} \!\! a^-(z) {d}z.$$ The total effect of competition on a particle is the sum of the effects of competition with all individual particles. 
\end{enumerate}
Here we have interacting particles, in contrast to the usual branching process.  One can expect physically that for arbitrary non-trivial competition ($a^-\in C(\mathbb{R}^d)$, $A^- > 0$), there will exist a limiting distribution of the particles.  At each site $x$, with population at time $t$ given by $n(t,x)$, three rates are relevant, the birth rate $b$ and mortality rate $\mu$, each proportional to $n(t,x)$ and the death rate due to competition, proportional to $n(t,x)^2$.   Heuristically, when $n(t,x)$ is small the linear effects will dominate, which means that if $b > \mu$ the population will grow.  As $n(t,d)$ becomes large, however, the quadratic effect will become inceasingly dominant, which will prevent unlimited growth.  At present, this fact has been proven only under strong restrictions on $a^+$ and $a^-$ \cite{fkkk2}.

\section{Mean Field Approximation}\label{sec:meanfield}
In this section, we introduce the mean field approximation to the general Bolker-Pacala process. For the remainder of the paper, we consider only the approximation model described below. All particles live on the lattice, $\mathbb{Z}^d$. We can suppose that each lattice point $x$ has an associated square $x+[0,1)^d$, and the number of particles at $x$ represents the number of inhabitants in the continuous model of that square that is associated with $x$.  Additionally, we assume that migration and immigration are uniform within the box. Let $Q_{L} \subset \mathbb{Z}^d$ be a box with $|Q_{L}| = L$, $L$ being a large parameter, and suppose that no particles exist outside of $Q_{L}$. We let
\begin{align*}
 a^+(x) = \frac{\kappa}{L} \text{ on } Q_L, \quad
 a^-(x) = \frac{\gamma}{L^2}  \text{ on } Q_L
\end{align*}
for some rates $\kappa, \gamma \ge 0$. Thus, the distribution of a particle after a jump due to migration or immigration is uniform on $Q_L$. We let $b$ and $\mu$ be the birth and mortality rates, respectively.
Let 
\begin{equation*}
 N_L(t) = \sum_{x\in Q_L} n(t,x)
\end{equation*}
be the total number of particles. $N_L(t)$ is a Markov process, which we call the ``logistic'' Markov chain.

The transition rates for $N_L(t)$ are 
\begin{displaymath}
  P\left(N_L(t+dt) = j | N_L(t) = n  \right) =  \left\{
     \begin{array}{ll}
       nb\,dt + o(dt^2)& \text{ if } j = n+1\\
       n\mu\,dt + \gamma n^2/L \,dt + o(dt^2)& \text{ if } j = n-1\\
       o(dt^2)& \text{ otherwise }
     \end{array}
   \right.
\end{displaymath} 
If $N_L(t)$ is large, therefore, there is a left drift, whereas if $N_L(t)$ is small, there will be a drift to the right. An important point is the equilibrium point, $n_L^*$, where the rates are equal, that is,
\begin{equation*}
 b n_L^* = \mu n_L^* + \frac{\gamma n_L^{*2}}{L},
\end{equation*}
which means that
\begin{equation*}
 n_L^* = \left\lfloor \frac{L(b-\mu)}{\gamma} \right\rfloor.
\end{equation*}
We will show that as $L$ becomes large, the Markov chain $N_L(t)$ tends quickly to a neighborhood of $n_L^*$ and afterward fluctuates randomly around $n_L^*$. Our goal is the analysis of these fluctuations.

The logistic Markov chain is a particular case of a birth and death random walk on $\mathbb{Z}^1_+ = \{0,1,\ldots\}$, which has been studied extensively (e.g. \cite{fe1, km65, tk70}). The generator for a birth and death random walk, $X(t)$, $t\geq 0$ is given by 
\begin{equation*}
\mathcal{L}\psi (x) = \alpha_x \psi(x-1) - (\alpha_x + \beta_x)\psi(x) + \beta_x \psi(x+1)
\end{equation*}
for $x > 0$ and 
\begin{equation*}
\mathcal{L}\psi (0) = \beta_0\psi(1) -  \beta_0 \psi(0).
\end{equation*}
For the logistic Markov chain $X(t) := N_L(t)$, the transition rates are
\begin{align*}
 \beta_x = bx, x\geq 1, \beta_0 = 1, \quad
  \alpha_x = \mu x + \frac{\gamma x^2}{L}, x\geq 1
\end{align*}
We will analyze the logistic Markov chain asymptotically as $L\to\infty$.
In fact, it is more convenient to study a modified logistic chain with transition rates:
\begin{align}
 \beta_x = b(x+1), x\geq 0, \quad
  \alpha_x = \mu x + \frac{\gamma x^2}{L}, x\geq 1.
\label{modifiedTrnsRates}
\end{align}
Here, the equilibrium point is

\begin{equation*}
 \tilde{n}_L^* = \left\lfloor \frac{(b-\mu) + \sqrt{(b-\mu)^2 + 4\gamma/L}}{2\gamma/L} \right\rfloor,
\end{equation*}
is equal to the old equilibrium point for large enough $L$. For convenience, when appropriate, we assume that $L$ is such that $(b-\mu)L/\gamma$ is an integer. Thus, $$\tilde{n}_L^* = \frac{(b-\mu)L}{\gamma}$$.

There are two reasons for this change. First, unlike the unmodified chain, the modified logistic chain has no absorbing state at $x = 0$. Second, calculations for the modified logistic chain are simpler. Let us stress that as $L$ becomes large, the asymptotics are equivalent for both models. 


\section{Detailed Analysis of the Logistic Markov Chain}


\subsection{Hypergeometric Functions}

In this section, we will study the asymptotic properties of the degenerated hypergeometric function. For a particular choice of parameters, such functions will play an important role in the analysis of the logistic Markov chain with transition rates \eqref{modifiedTrnsRates}.

The classical degenerated (or confluent) hypergeometric function depends on two parameters, $\alpha$ and $\gamma$, and is given by the Taylor expansion

\begin{equation*}
 \Phi(\alpha,\gamma,z) = 1 + \frac{\alpha}{\gamma}\frac{z}{1!} + \frac{\alpha(\alpha + 1)}{\gamma (\gamma + 1)}\frac{z^2}{2!} + \ldots + \frac{\alpha (\alpha + 1)\ldots (\alpha + n-1)}{\gamma (\gamma +1)\ldots (\gamma +n-1)}\frac{z^n}{n!} + \ldots
\end{equation*}
(see \cite{gr65}, 9.21). $\Phi$ is an entire function of order $1$.
We focus on the special case when $\alpha = 1$ and $\gamma = A$, $A \gg 1$, and call such a function

\begin{equation*}
 \mathcal{F}(A,z) = 1 + \frac{z}{A} + \ldots + \frac{z^n}{A(A+1)\ldots(A +n - 1)} + \ldots
\end{equation*} 
In our case, the general integral representation of $\Phi(\alpha,\gamma,z)$ (\cite{gr65}, 9.211,1) leads to the formula

\begin{equation*}
 \mathcal{F}(A,z) = \frac{2^{1-A}e^{z/2} \Gamma(A)}{\Gamma(A-1)}\int\limits_{-1}^{1} (1-t)^{A-2}e^{zt/2}\,dt = 2^{1-A}e^{z/2} (A-1) \int\limits_{-1}^{1} (1-t)^{A-2}e^{zt/2}\,dt
\end{equation*} 
The substitution $t=1-s$ leads to
\begin{align}
 &\mathcal{F}(A,z) = 2^{1-A}\,e^{z/2}\,(A-1) \int_{0}^{2} s^{A-2}e^{z(1-s)/2}\, ds \notag\\
& = 2^{1-A}\,e^{z}\,(A-1) \int_{0}^{2} s^{A-2}e^{-zs/2}\, ds 
= \frac{e^{z}\,(A-1)}{z^{A-1}}\int_{0}^{z} t^{A-2}e^{-t}\, dt  \label{F-final}
\end{align}
where, in the last step, we used the substitution $sz/2 = t$.
The integral factor in \eqref{F-final} is the incomplete $\Gamma$-function:
\begin{equation*}
 \gamma(A-1,z) = \int_{0}^{z}t^{A-2}e^{-t}\, dt
\end{equation*}
Of course, 
\begin{equation*}
 \gamma(A-1,z) = \Gamma(A-1)-\Gamma(A-1,z), \qquad  \Gamma(A-1,z) = \int_{z}^{\infty} t^{A-2}e^{-t} \, dt
\end{equation*}
If $A,z \gg 1$, one can use the Laplace method to obtain asymptotics for $\mathcal{F}(A,z)$. Putting $t=As$, we have

\begin{equation*}
 \gamma(A-1,z) = A^{A-1} \int_{0}^{z/A} e^{-A\left(s -\frac{A-2}{A} \ln s\right)}\, ds
\end{equation*}
The critical point of the ``phase function'' $$\Phi(s) = s - \frac{A-2}{A}\ln s$$ (where $\Phi'(s_0) = 0$) is $s_0 = (A-2)/A$ and the asymptotic behavior of $\gamma(A-1,z)$ depends on the relationship between $z/A$ and $s_0$.

Routine application of the Laplace method leads to the theorem
\begin{theorem}\label{Degen}
 Consider
\begin{equation*}
 \mathcal{F}(A,z) = 1 + \frac{z}{A} + \ldots + \frac{z^n}{A(A+1)\ldots(A + n - 1)} + \ldots.
\end{equation*}
If $A,z\to\infty$ and
\begin{enumerate}
 \item[I.] $z<A$, $\sqrt{A}/(A-z)\to 0$, 
then

\begin{equation*}
 \mathcal{F}(A,z) \sim \frac{A}{A-z} = o(\sqrt{A})
\end{equation*}

\item[II.] $z>A$ and  $\sqrt{A}/(z-A)\to 0$, then (using Stirling's formula)

\begin{equation}\label{DegII}
 \mathcal{F}(A,z) \sim \frac{e^z \Gamma(A)}{z^{A-1}} \sim e^{z-A+1} \left( \frac{A-1}{z} \right)^{A-1} \sqrt{2 \pi A}
\end{equation}

\item[III.] $z = A + h\sqrt{A}, \quad h > 0$ is a constant, then

\begin{equation*}
 \mathcal{F}(A,z) \sim e^{-h^2/2} \Phi(h)\sqrt{2\pi A}, \qquad \Phi(h) = \frac{1}{\sqrt{2\pi}}\int_{-\infty}^{h}e^{-y^2/2}dy
\end{equation*}

\item[IV.] $z = A - h\sqrt{A}, \quad h > 0$ is a constant, then

\begin{equation*}
 \mathcal{F}(A,z) \sim e^{h^2/2}\Phi(-h)\sqrt{2\pi A}
\end{equation*}

\end{enumerate}

\end{theorem}


\subsection{General Results for 1D Random Walks}
Consider a general random walk $X(t)$ on $\mathbb{Z}^{1}_{+} = \{ 0, 1, 2, \ldots \}$ in continuous time and the transition rates from $x$, $\alpha_x$ to the left and $\beta_x$ to the right, where $\beta_x > 0, x\geq 0$ and $\alpha_x>0, x\geq1$ and $\alpha_0 = 0$.

The general birth and death random walk with generator $$\mathcal{L}\psi (x) = \alpha_x \psi(x-1) - (\alpha_x + \beta_x)\psi(x) + \beta_x \psi(x+1)$$ is ergodic if and only if the following series converges (\cite{fe1}, XV.7)
\begin{equation}\label{Seqn}
S = 1 + \frac{\beta_0}{\alpha_1} + \frac{\beta_0\beta_1}{\alpha_1\alpha_2} + \ldots + \frac{\beta_0\ldots\beta_n}{\alpha_1\ldots\alpha_{n+1}} + \ldots.
\end{equation}
In this case, the invariant distribution is given by (\cite{fe1}, XV)
\begin{equation}\label{Invteqn}
\pi(x) = \lim_{t\to\infty} p(t,\cdot,x) = \left\{
     \begin{array}{ll}
        S^{-1}, & x=0\\
       S^{-1}\frac{\beta_0\ldots\beta_{x-1}}{\alpha_1\ldots\alpha_{x}}, & x>0
           \end{array}
   \right.
\end{equation}

Let $\tau_y = \min\{t : x(t) = y \}$. From the ergodicity of $X(t)$, it follows that for any $x,y\in \mathbb{Z}^{1}_{+}, x\neq y$, $E_x \tau_y$ is finite. We will use the notations $E_x\tau_y = E\tau_{x\to y}$ and $E_x\tau_y = u(x,y)$. For fixed $y$ and $x>y$, $u(x,y)$ satisfies 

\begin{displaymath}
\left\{
     \begin{array}{ll}
        \mathcal{L}u(x,\cdot) = -1, & x>y\\
       u(y,\cdot) = 0
           \end{array}
   \right.
\end{displaymath} 
One can understand $u(x,\cdot)$ as $\lim\limits_{L\to\infty} u_L(x,y), x\in [y,L]$, where
\begin{displaymath}
\left\{
     \begin{array}{l}
        \mathcal{L}u _L(x,\cdot) = -1,\\
       u_L(y,\cdot) = u_L(L,\cdot) = 0
           \end{array}
   \right.
\end{displaymath}
To see this, for finite $L>x$ let $\tau_{y,L} = \min\{t:x(t) = y\ \mathrm{or}\ L\}$.  Again, $E_x \tau_{y,L}$ is finite. Set $u_L(x,y) = E_x\tau_{y,L}$. For fixed $y$ and $x>y$, $u_L(x,y)$ satisfies 
\begin{displaymath}
\left\{
     \begin{array}{ll}
        \mathcal{L}u_L(x,\cdot) = -1, & x>y\\
        u_L(y,y) = u_L(L,L) = 0
           \end{array}
   \right.
\end{displaymath} 
By the Monotone Convergence Theorem, then, $\tau_{y,L} \to \tau_y$ as $L \to \infty$ and so $E_x \tau_{y,L} \to E_x \tau_y$, as $L \to \infty$.

Note that for $x>y$,
\begin{equation*}
\tau_{x\to y} = \tau_{x\to x-1} + \tau_{x-1 \to x-2} + \ldots \tau_{y+1\to y}
\end{equation*}

\begin{lemma}\label{LemEtau} Assume that the random walk $X(t)$ is ergodic. Then
\begin{equation}\label{EqnEtau}
\mathrm{E} \tau_{y+1 \to y} = \frac{1}{\alpha_{y+1}}\left[1 + \frac{\beta_{y+1}}{\alpha_{y+2}} + \frac{\beta_{y+1}\beta_{y+2}}{\alpha_{y+2}\alpha_{y+3}} +\ldots \right]
\end{equation}
\end{lemma}

\begin{proof}
Let $y$ be fixed. Set $$u(x-y) := \mathrm{E}\tau_{x\to y},\quad x \ge y.$$ Then $u(0) = 0$ and

\begin{equation}
 \alpha_x u(x-1) - (\alpha_x + \beta_x)u(x) + \beta_xu(x+1) = -1.
\label{uEqn}
\end{equation}
Set $$\Delta(x) = u(x) - u(x-1), x\geq 1.$$ Then

\begin{equation}
 \Delta(x+1) = \frac{\alpha_x}{\beta_x}\Delta(x) - \frac{1}{\beta_x}
\label{DeltaEqn}
\end{equation}
Consider a random walk $\tilde{X}_N(t)$ with the same transition rates but on the finite interval $[0,N]$. We are interested in the first entrance to $N$. Let $$\tau_{0,N} = \min\{t: \tilde{X}_N(t) = 0 \text{ or }N\},  \quad u_N(x) := \mathrm{E}_x(\tau_{0,N}), \quad\text{and}\quad$$ $$\Delta_N(x) := u_N(x) - u_{N}(x-1).$$ Then, $u_N$  satisfies \eqref{uEqn} with boundary conditions $$u_N(0) = u_N(N) = 0$$ and $\Delta_N(x)$ satisfies \eqref{DeltaEqn} as well as $\lim\limits_{N\to\infty}\Delta_N(x) = \Delta(x)$, pointwise in $x$.

For the finite chain,
\begin{equation}
 \Delta_N(1) + \Delta_N(2) + \ldots + \Delta_N(N) = 0
\label{SumofDeltas}
\end{equation}
From \eqref{DeltaEqn} and \eqref{SumofDeltas}, one obtains
\begin{equation*}
\Delta_N(k) = \frac{\alpha_1\alpha_2\ldots\alpha_{k-1}}{\beta_1\beta_2\ldots\beta_{k-1}}\Delta_N(1) - \frac{\alpha_2\ldots\alpha_{k-1}}{\beta_1\beta_2\ldots\beta_{k-1}} - \frac{\alpha_3\ldots\alpha_{k-1}}{\beta_2\ldots\beta_{k-1}} -\ldots - \frac{1}{\beta_{k-1}}
\end{equation*}
Note that $$\Delta_N(1) = u_N(1) - u_N(0) = u_N(1).$$
We now multiply through by $\frac{\beta_1\beta_2\ldots\beta_{k-1}}{\alpha_1\alpha_2\ldots\alpha_{k-1}}$ to obtain equations for $\Delta_N(1)$

\begin{equation*}
 \Delta_N(1) = \frac{\beta_1\beta_2\ldots\beta_{k-1}}{\alpha_1\alpha_2\ldots\alpha_{k-1}}\Delta_N(k) + \frac{\beta_1\beta_2\ldots\beta_{k-2}}{\alpha_1\alpha_2\ldots\alpha_{k-1}} + \frac{\beta_1\beta_2\ldots\beta_{k-3}}{\alpha_1\alpha_2\ldots\alpha_{k-2}} + \ldots + \frac{1}{\alpha_1}
\end{equation*}
This is true for all $k$, including $k=N$, and so
\begin{equation*}
 \Delta(1) := \lim_{N\to\infty}\Delta_N(1) = \frac{1}{\alpha_1}\left(1 + \frac{\beta_1}{\alpha_2} + \frac{\beta_1\beta_2}{\alpha_2\alpha_3} + \ldots \right).
\end{equation*}
Therefore, $E\tau_{y+1\to y}$ is given by
\begin{equation*}
E\tau_{y+1\to y} = \Delta(y+1) = \frac{1}{\alpha_{y+1}}\left(1  + \frac{\beta_{y+1}}{\alpha_{y+2}} + \frac{\beta_{y+1}\beta_{y+2}}{\alpha_{y+2}\alpha_{y+3}} + \ldots \right).
\qedhere\end{equation*}
\end{proof}

The second factor of \eqref{EqnEtau} has the following interpretation.  Put $$S_y := 1 + \frac{\beta_{y}}{\alpha_{y+1}} + \frac{\beta_{y}\beta_{y+1}}{\alpha_{y+1}\alpha_{y+2}} +\ldots$$ (i.e., $S = S_0 = 1 + \frac{\beta_{0}}{\alpha_{1}} + \frac{\beta_{0}\beta_{1}}{\alpha_{1}\alpha_{2}} +\ldots$). Then, 
\begin{equation}\label{Invteqn2}
\pi(x) = \lim_{t\to\infty} p(t,\cdot,x) = S_y^{-1}\frac{\beta_0\ldots\beta_{x-1}}{\alpha_1\ldots\alpha_{x}}.
\end{equation}
$S_y^{-1} = \pi_y^{(y)}$ is, therefore, the invariant probability for the random walk $X(t)$ on $[y, \infty)$ with reflection rate $\beta_y$ at $y$.

Considering the cycles between successive transitions $y+1 \to y$ and $y \to y+1$ and applying the Law of Large Numbers we deduce that
\begin{equation*}
\pi_y^{(y)} = \frac{1/\beta_y}{1/\beta_y + \mathrm{E}\tau_{y+1\to y}}
\end{equation*}
and so
\begin{equation*}
S_y = 1 + \beta_y \mathrm{E}\tau_{y+1\to y}.
\end{equation*}
It also follows that
\begin{equation*}
\mathrm{E}\tau_{y+1\to y} = \frac{1}{\alpha_{y+1}} S_{y+1}.
\end{equation*}

For the logistic chain with $$\beta_x = b\left(x+1\right), \quad \alpha_x = x\left(\mu + \frac{\gamma x}{L}\right),$$ we represent $S_y$, $y \ge 0$ in terms of the hypergeometric function $F(A,z)$.

\begin{equation}
\begin{split}
S_y &= 1 + \frac{\beta_y}{\alpha_{y+1}} + \frac{\beta_y \beta_{y+1}}{\alpha_{y+1} \alpha_{y+2}} + \ldots \\
	&= 1 + \frac{b}{\mu + \frac{\gamma (y+1)}{L}} + \frac{b^2}{\left (\mu + \frac{\gamma (y+1)}{L} \right ) \left (\mu + \frac{\gamma (y+2)}{L}\right )} + \ldots \\
	&= 1 + \frac{bL/\gamma}{\frac{\mu L}{\gamma} + y+1} + \frac{ (bL/\gamma)^2}{\left (\frac{\mu L}{\gamma} + y+1 \right ) \left(\frac{\mu L}{\gamma} + y+2 \right)} + \ldots \\
	&= F\left(\mu L/\gamma + y, bL/\gamma \right).
\end{split}
\end{equation}
This formula, asymptotics from Theorem \ref{Degen}, and the obvious relation
\begin{equation*}
	\mathrm{E}\tau_{x\to y} = \mathrm{E}\tau_{x\to x-1} + \mathrm{E}\tau_{x-1\to x-2} + \cdots + \mathrm{E}\tau_{y+1\to y}, \quad \text{ for } x>y
\end{equation*}
will be the basis for the analysis of the first passage times.

Let us note also that there is a recurrence formula connecting $\mathrm{E}\tau_{y+1\to y}$ and $\mathrm{E}\tau_{y+2\to y+1}$.  From \eqref{EqnEtau}, we obtain
\begin{equation*}
\begin{split}
	\mathrm{E}\tau_{y+1\to y} &= \frac{1}{\alpha_{y+1}} + \frac{\beta_{y+1}}{\alpha_{y+1}} \left[\frac{1}{\alpha_{y+2}} \left(1 + \frac{\beta_{y+2}}{\alpha_{y+3}} + \frac{\beta_{y+2}\beta_{y+3}}{\alpha_{y+3}\alpha_{y+4}} +\ldots \right) \right] \\
	&= \frac{1}{\alpha_{y+1}} + \frac{\beta_{y+1}}{\alpha_{y+1}} \mathrm{E}\tau_{y+2\to y+1}.
\end{split}
\end{equation*}


\subsection{Limit theorems for the invariant distribution of the logistic Markov chain}

We apply the general results on the 1D ergodic random walk on $\mathbb{Z}^{1}_{+}$ to the particular case of the logistic Markov chain, which we defined above as a means of studying the mean field Bolker-Pacala process. For the modified chain, by the results above, we can obtain a local Central Limit Theorem.

\begin{theorem}[Local CLT]
 Let $b > \mu$. If $k = O(L^{2/3})$, then, for the invariant distribution $\pi_L$,

\begin{equation*}
 \pi_L(n_L^* + k) \sim \frac{e^{-k^2/2\sigma_L^2}}{\sqrt{2\pi\sigma_L^2}},
\end{equation*}
where $\sigma_L^2 = Lb/\gamma$.

\end{theorem}

\begin{proof}

From \eqref{Seqn} we have
\begin{equation}\label{Seqn2}
 S = 1 + \frac{b}{\mu + \gamma/L} + \frac{b^2}{(\mu + \gamma/L)(\mu + 2\gamma/L)} + \ldots + \frac{b^n}{(\mu + \gamma/L)\ldots(\mu + n\gamma/L)} + \ldots
\end{equation}

And for $n\in\mathbb{Z}^1_+$, from \eqref{Invteqn}

\begin{equation}\label{Invteqn2}
 \pi_L(n) = S^{-1}\frac{b^n}{(\mu + \gamma/L)\ldots(\mu + n\gamma/L)}.
\end{equation}
To analyze this series, we consider the position, which we label $n_L^*$, such that the ratio of the terms $\frac{\pi_L(n)}{\pi_L(n-1)} \approx 1$.  (If there are two such positions, we take $n_L^*$ to be the larger.)  Taking this ratio, we find that

\begin{equation*}
\frac{(b-\mu)L}{\gamma} -1 \leq n_L^* \leq \frac{(b-\mu)L}{\gamma}.
\end{equation*}
We obtain the formula

\begin{equation*}
 \pi_L(n_L^* + k) = \pi_L(n_L^*)\cdot\frac{b^k}{A^k}\cdot\frac{1}{(1+\frac{\gamma}{LA})\ldots(1+\frac{k\gamma}{LA})},
\end{equation*}
where $A = \mu + \frac{n_L^*\gamma}{L} = b + O(1/L)$. Now, we consider

\begin{equation*}
 \sum_{i=0}^{k} \ln\left(1 + \frac{\gamma i}{AL}\right) = \int\limits_{0}^{k} \ln \left(1 + \frac{\gamma x}{AL}\right) \,dx + O\left(\ln\left(1 + \gamma k/AL\right)\right).
\end{equation*}
We integrate the series $\ln(1+x) = x - \frac{1}{2}x^2 + \frac{1}{3}x^3 - \cdots$, and take $k = O\left(L^{2/3}\right)$.

\begin{align*}
 \int\limits_{0}^{k} \ln\left(1 + \frac{\gamma x}{AL}\right)\, dx &= \int\limits_{0}^{k} \frac{\gamma x}{AL} \, dx - \frac{1}{2} \int\limits_{0}^{k} \left(\frac{\gamma x}{AL}\right)^2 \, dx\\
& = \frac{\gamma k^2}{2AL} - \frac{1}{6}\cdot\frac{\gamma^2 k^3}{A^2L^2} + \ldots\\
& = \frac{\gamma k^2}{2AL} + O\left(1 \right).
\end{align*}
Thus, 

\begin{equation*}
 \pi_L(n_L^* + k) \approx \pi_L(n_L^*)\left(\frac{b}{b+O(1/L)} \right)^ke^{-\gamma k^2/2LA}
\end{equation*}
It remains to calculate $\pi_L(n_L^*)$, for which we use \eqref{Invteqn2}.

We apply the asymptotic formulas for $\mathcal{F}(A,z)$ (\eqref{DegII}) to \eqref{Seqn2} to calculate $S$:

\begin{equation*}
\begin{split}
	 S_L &= 1 + \frac{(bL/\gamma)}{(1 + \mu L/\gamma)} + \frac{(bL/\gamma)^2}{(1 + \mu L /\gamma)(2 + \mu L /\gamma)} + \ldots \\
	& = \mathcal{F}\left(\mu L /\gamma, bL/\gamma \right) \\
	& \sim e^{n_L^*} \left(\frac{\mu}{b}\right) ^ {\mu L / \gamma} \sqrt{2\pi \frac{\mu}{\gamma} L}.
\end{split}
\end{equation*}
For the product in the denominator, we use the Euler-Maclaurin formula (see \cite{WW1915}). First,

\begin{equation*}
\begin{split}
\Pi_{n_L^*} := (\mu + \gamma/L)\cdots &(\mu + n\gamma/L) = \mu^{n_L^*}(1+\frac{\gamma}{\mu L})\cdots(1+\frac{n_L^*\gamma}{\mu L}) \\
& = \mu^{n_L^*} \exp \left(\sum_{k=0}^{n_L^*}\ln (1 + k \omega) \right),
\end{split}
\end{equation*}
where $\omega := \frac{\gamma}{\mu L}$.
The Euler-Maclaurin formula gives
\begin{equation}\label{EMf}
\begin{split}
	\sum_{k=0}^r \ln{(1+\omega k)} &= \frac{1}{\omega} \int_{1}^{1+r\omega} \!\! \ln x \,{d}x + \frac{1}{2} \ln (1+r\omega) + O(\omega)\\
	&= (\frac{1}{\omega} + r + \frac{1}{2}) \ln (1 + r\omega) - r + O(\omega).
\end{split}
\end{equation}
Thus
\begin{equation*}
\begin{split}
\Pi_{n_L^*} &\sim \mu^{n_L^*} \exp \left[(\frac{1}{\omega} + r + \frac{1}{2}) \ln (1 + r\omega) - r \right] \\
& = \frac{\mu^{n_L^*}\left(\frac{b}{\mu}\right) ^{Lb/\gamma + 1/2}}{e^{L(b-\mu)/\gamma}}.
\end{split}
\end{equation*}
And so, finally,
\begin{equation}
\pi_L(n_L^*) = \frac{b^{n_L^*}}{S_L \Pi_{n_L^*}} \sim \frac{1}{\sqrt{2\pi \frac{b}{\gamma}L}}.
\end{equation}
\end{proof}

\noindent Our large deviations result is
\begin{theorem}[Large Deviations]\label{LDT}
For constant $\delta > 0$,

\begin{equation*}
 \pi(n_L^* + \delta L) \asymp \frac{1}{\sqrt{L}} e^{-L f(\delta \gamma/ b) b /\gamma},
\end{equation*}
where $f(z) := \displaystyle\sum_{n=2}^\infty \frac{(-1)^n}{n(n-1)} z^n = \int_{0}^z \!\! \ln(1+x) {d}x$.

\end{theorem}

\begin{proof}
Set $N := n_L^*(1+\delta)$ and $k := \delta n_L^*$.
Applying the formula for general Markov chains (\ref{Invteqn}), the calculated invariant probability $\pi(n_L^*)$, and the Euler-Maclaurin formula (with $\omega:=\frac{\gamma}{bL}$), 
\begin{equation}
\begin{split}
\pi(N) &= \pi(n_L^*) \frac{\beta_{n_L^*}\cdots \beta_{n_L^*+k-1}}{\alpha_{n_L^*+1}\cdots \alpha_{n_L^*+k}} \sim \frac{1}{\sqrt{2\pi \frac{b}{\gamma}L}} \frac{b^k}{(\mu + (n_L^*+1)\frac{\gamma}{L})\cdots(\mu + (n_L^*+k)\frac{\gamma}{L})} \\
& = \frac{1}{\sqrt{2\pi \frac{b}{\gamma}L}} \frac{1}{(1+\omega)\cdots(1+k\omega)} \\
& = \frac{1}{\sqrt{2\pi \frac{b}{\gamma}L}} \exp \left(-\sum_{j=0}^{k}\ln (1 + j \omega) \right) \\
& \sim \frac{1}{\sqrt{2\pi \frac{b}{\gamma}L}} \exp \left(-[(\frac{1}{\omega} + k + \frac{1}{2}) \ln (1 + k\omega) - k] \right).
\end{split}
\end{equation}

Using the expansion $\ln(1+x) = x - \frac{1}{2}x^2 + \frac{1}{3}x^3  - \frac{1}{4}x^4 + \ldots$, substituting in for $\omega$, $k$, and $n_L^*$, and multiplying through, we obtain for the exponent: 
\begin{equation}
(\frac{1}{\omega} + k + \frac{1}{2}) \ln (1 + k\omega) - k = L \frac{b}{\gamma} \sum_{n=2}^\infty \frac{(-1)^n}{n(n-1)} \left(\frac{\delta(b-\mu)}{b} \right)^n.
\end{equation}

\end{proof}


\subsection{Global Limit Theorems}
A functional Law of Large Numbers for the logistic Markov chain follows directly from Theorem 3.1 in Kurtz (1970 \cite{tk70}).  Likewise, a functional Central Limit Theorem follows from Theorems 3.1 and 3.5 in Kurtz (1971 \cite{tk71}).  We state these theorems here, therefore, without proof.

Define a new stochastic process for the population density, $Z_L(t) := \frac{N_L(t)}{L}$.  Set $z^* = \frac{n_L^*}{L} = \frac{b - \mu}{\gamma}$. 

We define the transition function, $f_L(\frac{N_L}{L},j) := \frac{1}{L} p(N_L,N_L+j)$. Thus,
        \begin{equation*}
        f_{L}(z,j) = \left\{
        \begin{array}{l}
        \frac{b N_L}{L} = b z \hspace {3.3 cm} j=1 \\
        \frac{\mu N_L +\frac{\gamma}{L} N_L^2}{L} = \mu z + \gamma z^2 \hspace {1.4 cm} j = -1\\
        \mathrm{(not\, needed)} \hspace {2.9 cm} j = 0 \\
        \end{array}
        \right.
        \end{equation*}  
Note that $f_L(z,j)$ does not, in fact, depend on $L$ and we write simply $f(z,j)$.

\begin{theorem}[Functional LLN]
As $L \to \infty$, $Z_L(t) \to Z(t)$ uniformly in probability, where $Z(t)$ is a deterministic process, the solution of
	\begin{equation}\label{cm1}
	\frac{d Z(t)}{dt} = F(Z(t)), \ \ \ \ \ Z(0) = z_0. 
	\end{equation}
where 
	\begin{equation*}
   F(z) := \displaystyle\sum_{j} jf(z,j) = b z - \mu z - \gamma z^2 = \gamma z(z^* - z).
	\end{equation*}
\end{theorem}
Equation \eqref{cm1} is in fact that of the stochastic logistic model, as discussed by Pollett \cite{pp00}.  It has the solution
\begin{equation*}
Z(t,z) = \frac{z^*z}{z+(z^*-z)e^{-\gamma z^* t}},\ t \ge 0.
\end{equation*}
Next, define $G_L(z) := \displaystyle\sum_{j}j^2 f_L(z,j) = (b + \mu)z + \gamma z^2$.  This too does not depend on $L$ and we simply write $G(z)$.

\begin{theorem}[Functional CLT]
If $\sqrt{L} \left(Z_L(0)-z^* \right) = \zeta_0$, the processes 
\[\zeta_L(t):=\sqrt{L}(Z_L(t) - Z(t)) \]
converge weakly in the space of cadlag functions on any finite time interval $[0, T]$ to a Gaussian diffusion $\zeta(t)$ with:
\begin{enumerate}
\item[1)] initial value $\zeta(0) = \zeta_0$, 

\item[2)] mean $$E\zeta(s) = \zeta_0 L_s := \zeta_0 e^{\int\limits_0^s F'(Z(u))du},$$

\item[3)] variance $$\mathrm{Var}(\zeta(s)) = L_s^2 \int\limits_0^s L_u^{-2} G(Z(u))du.$$
\end{enumerate}

Suppose, moreover, that $F(z_0) = 0$, i.e., $z_0 = z^*$, the equilibrium point. Then, $Z(t) \equiv z_0$ and $\zeta(t)$ is an Ornstein-Uhlenbeck process (OUP) with initial value $\zeta_0$, infinitesimal drift $$q := F'(z_0) = \mu - b$$ and infinitesimal variance $$a := G(z_0) = \frac{2b}{\gamma}(b - \mu).$$

Thus, $\zeta(t)$ is normally distributed with mean $$\zeta_0 e^{qt} = \zeta_0 e^{-(b - \mu)t}$$ and variance $$\frac{a}{-2q}\left(1 - e^{2qt}\right) = \frac{b}{\gamma}\left(1 - e^{-2(b - \mu)t}\right).$$
\end{theorem}

Finally, as $L \to \infty$, the first exit time $\tau_{LA}$ from $n^*$ to $n^* \pm A\sqrt{L}$ will be approximately the first exit time $\tau_A$ for the OUP starting at 0 from $\pm A$.  The distribution of $\tau_A$ has been shown by Breiman \cite{LB66} to be
	\begin{equation}
	P\{\tau_A > t\} = \alpha e^{-2\nu(A) t} + O\left(e^{-2(\nu(A)+\delta) t}\right),
	\end{equation}
such that
        \begin{equation*}
        \begin{array}{l}
        \hspace {.5 cm} (i) \lim \limits_{A \rightarrow \infty} \nu(A) = 0\\
\\
        \hspace {.5 cm} (ii) \lim \limits_{A \rightarrow 0} \nu(A) = \infty\\
\\
        \hspace {.5 cm} (iii) \mathrm{\,if} A^2 \mathrm{\,is\,the \,smallest\,positive\,root\,of} \\
	\hspace {3 cm} \displaystyle\sum_{k=0}^m \frac{(-2A^2)^k}{(2k)!} \frac{m!}{(m-k)!}, \\
	\hspace {1.3 cm} \mathrm{then\,} \nu(A) = m. \\
       \end{array}
        \end{equation*}  


\subsection{First passage time}
As $L \to \infty$,
\begin{equation*}
\begin{split}
\mathrm{E}\tau_{n^* \to 0} &= \sum_{k=1}^{n^*} \tau_{k \to k-1} = \sum_{k=1}^{n^*} \frac{S_k}{\alpha_k} = S_1 \sum_{k=1}^{n^*}\frac{1}{\alpha_k} \frac{S_k}{S_1} \\
	&\sim S_1 \frac{1}{\mu} \left(1 + \frac{1}{2}\frac{\mu}{b} + \frac{1}{3} \left(\frac{\mu}{b}\right)^2 + \ldots \right) \\
	 &= \frac{b}{\mu^2} \ln \left(\frac{b}{b-\mu}\right) S_1.
\end{split}
\end{equation*} 
Let $n_1 = (1 - \delta_1) n^*$, $\delta_1 > 0$.  Then
\begin{equation}\label{tilde1}
\begin{split}
&\mathrm{E}\tau_{n^* \to n_1} = \sum_{k=n_1+1}^{n^*} \tau_{k \to k-1} = \sum_{k=n_1+1}^{n^*} \frac{S_k}{\alpha_k} \\
	&\sim \frac{S_{n_1+1}}{\alpha_{n_1+1}} \left(1 + \frac{n_1+1}{n_1+2}\rho_1 + \frac{n_1+1}{n_1+3} \rho_1^2 + \ldots \right) \sim \frac{S_{n_1+1}}{\alpha_{n_1+1}} \cdot \frac{1}{1-\rho_1},
\end{split}
\end{equation} 
where $\rho_1 := 1- (1 - \frac{\mu}{b}) \delta_1$.

Next, we analyze the first passage time $u(x)$ to $\{n_1,n_2\}$ for $x\in [n_1,n_2]$, where $n_1 = (1-\delta_1)n^*$, $n_2 = (1+\delta_2)n^*$, $\delta_1, \delta_2 > 0$. $u$ must satisfy 
\begin{equation}\label{DE1}
\mathcal{L}u = -1,
\end{equation}
 as well as the boundary conditions $u(n_1) = u(n_2) = 0$.  A particular solution to \eqref{DE1} is $\psi_1(x)  := \mathrm{E}\tau_{x\to0}$.
Then, $$u = \psi_1 + c_1 + c_2\psi_2,$$ where $c_1$ and $c_2$ are constants and $\psi_2$ satisfies $\mathcal{L}\psi_2 = 0$.

We choose $\psi_2$ to satisfy $\psi_2(n^*) = 0$ and $\psi_2(n^*+1) = 1$.  This gives
\begin{equation}\label{psi2}
  \psi_2(x) =  \left\{
     \begin{array}{ll}
       -\left(\frac{\beta_{n^*}}{\alpha_{n^*}} + \frac{\beta_{n^*}\beta_{n^*-1}}{\alpha_{n^*}\alpha_{n^*-1}} + \ldots + \frac{\beta_{n^*} \cdots \beta_{x+1}} {\alpha_{n^*} \cdots \alpha_{x+1}} \right) & \text{ if } x < n^*\\
       0 & \text{ if } x = 0\\
       1 + \frac{\alpha_{n^*+1}}{\beta_{n^*+1}} + \ldots + \frac{\alpha_{n^*+1} \cdots \alpha_{x-1}} {\beta_{n^*+1} \cdots \beta_{x-1}} & \text{ if } x > n^*
     \end{array}
   \right.
\end{equation}
Let us find the asymptotics of $\psi_2(x)$ for $x=n_1=n^*(1-\delta_1)$, $x=n_2=n^*(1+\delta_2)$.  We will calculate up to a constant factor and use the standard notation, $a_n \asymp b_n$ to mean $$0 < c_1 \le \frac{a_n}{b_n} \le c_2 < \infty$$ for appropriate constants $c_1$ and $c_2$.

The last term
	\[ A_{n_2-1} := \frac{\alpha_{n^*+1}\cdots\alpha_{n_2-1}}{\beta_{n^*+1} \cdots \beta_{n_2-1}} \]
gives the main contribution to $\psi_2(n_2)$.  The terms $A_{n_2-k}$ asymptotically form a geometric progression with the ratio $$\rho_2 = \frac{1}{1+\delta_2(1-\mu/b)} < 1$$
(because
\begin{equation*}
	\begin{split}
	\frac{\alpha_{n_2-k}}{\beta_{n_2-k}} &= \frac{\mu(n_2-k) + \frac{\gamma}{L}(n_2-k)^2}{b(n_2-k+1)} = \frac{\mu}{b} + \frac{\gamma}{bL}n_2 + o(1) \\
	& = 1 + \delta_2(1- \frac{\mu}{b}) + o(1).)
	\end{split}
\end{equation*}
This means that, for $\omega := \frac{\gamma}{bL}$, 
	\[\psi_2(n_2) \sim A_{n_2-1}(1 + \rho_2 + \rho_2^2 + \ldots) \sim \frac{A_{n_2-1}}{1-\rho_2} \asymp A_{n_2-1}. \]
Moreover, using the Euler-Maclaurin formula (compare \eqref{EMf})
\begin{equation*}
	A_{n_2-1} \asymp \prod_{k=0}^{\delta_2  n^*} (1+k\omega) = \exp\left(\sum _{k=0}^{\delta_2  n^*} \ln(1+k\omega) \right) 
	\asymp \exp \left({\frac{1}{\omega} \int\limits_{0}^{\delta_2 (1-\frac{\mu}{b})} \!\! \ln (1+x) dx}\right),
\end{equation*}
and so
\begin{equation*}
	\psi_2(n_2) \asymp \exp \left({\frac{1}{\omega} \int\limits_{0}^{\delta_2 (1-\frac{\mu}{b})} \!\! \ln (1+x) dx}\right).
\end{equation*}
Similar calculations for $\psi_2(n_1)$ yield
\begin{equation*}
	\psi_2(n_1) \asymp -\exp \left(-{\frac{1}{\omega} \int\limits_{0}^{\delta_1 (1-\frac{\mu}{b})} \!\! \ln (1-x) dx}\right).
\end{equation*}

It is convenient to introduce some sort of ``symmetry'' of the logistic Markov chain with respect to the equilibrium point $n^* = \left\lfloor \frac{(b-\mu)L}{\gamma} \right\rfloor$. The simplest way to do so is to assume that 
	\[\mathrm{P}_{n^*}\{N_L(\tau_{[n_1,n_2]})=n_2\} \approx \frac{1}{2}. \]
This means that $\psi_2 (n_2) \asymp -\psi_2 (n_1)$, i.e.
	\[-\int\limits_{0}^{\delta_1 (1-\frac{\mu}{b})} \!\! \ln (1-x) dx =  \int\limits_{0}^{\delta_2 (1-\frac{\mu}{b})} \!\! \ln (1+x) dx. \]
The last equation uniquely determines $\delta_2$ as a function of $0<\delta_1<1$.

With this symmetry, we can determine $\mathrm{E}\tau_{n^*\to\{n_1,n_2\}}$.  Consider again the problem
\begin{equation}\label{DE2}
	\mathcal{L}u = -1,\quad u(n_1) = u(n_2) = 0. 
\end{equation}
We modify the previous particular solution and define $\tilde{\psi}_1(x) := \mathrm{E}\tau_{x\to n_1}$ for $x>n_1$.  Then, 
	\[u(x) := \tilde{\psi}_1(x) + c_1 + c_2 \psi_2(x) \]
satisfies \ref{DE2} for some constants $c_1$ and $c_2$.  Using $\psi_2 (n_2) \asymp -\psi_2 (n_1)$ and $\tilde{\psi}_1(n_1) = 0$, we obtain $c_1 = -\frac{1}{2}\tilde{\psi}_1(n_2)$.  From \ref{tilde1} and using I of Theorem \ref{Degen} we have that 
\begin{equation*}
\begin{split}
	\tilde{\psi}_1(n_2) &= \mathrm{E}\tau_{n_2\to n_1} = \frac{S_{n_1+1}}{\alpha_{n_1+1}(1-\rho_1)} \\
	&\sim \frac{\sqrt{2\pi\gamma}\exp\left(\frac{b}{\gamma}L\ln \rho_1 + \delta_1(1- \ln \rho_1)n^* \right)}{(b-\mu)(1-\delta_1)(1-\rho_1)\sqrt{b\rho_1L}} \\
	&\asymp \frac{\exp\left(\frac{b}{\gamma}L\ln \rho_1 + \delta_1(1- \ln \rho_1)n^* \right)}{\sqrt{L}}.
\end{split}
\end{equation*}
And so, finally
\begin{equation}\label{symmFPT}
\begin{split}
	\mathrm{E}\tau_{n^*\to\{n_1,n_2\}} &= u(n^*) = \tilde{\psi}_1(n^*) - \frac{1}{2}\tilde{\psi}_1(n_2) \sim \frac{1}{2}\tilde{\psi}_1(n_2) \\
	&\asymp \frac{\exp\left(\frac{b}{\gamma}L\ln \rho_1 + \delta_1(1- \ln \rho_1)n^* \right)}{\sqrt{L}}.
\end{split}
\end{equation}

This result tells us that arrival at 0 is extremely improbable and, therefore, we can essentially ignore it.  For the central limit theorem and large deviation results of the previous sections, therefore, we do not need to assume ergodicity, e.g., by adjusting the $\beta$ parameter to prevent absorption at 0.  This also suggests that the expected recurrence time to a point $k$, $\tau_k \sim \pi(k)^{-1}$, with $\pi(k)$ the invariant probability of state $k$.  In particular, the expected recurrence time to the equilibrium point $n_L^*$ will be $O(\sqrt{L})$ and to a point $n_L^* + \delta L$ will be $\sqrt{L}e^{O(L)}$.\\


\section{Conclusion}
The central mathematical topic concerning the Bolker-Pacala model is the existence (for $\beta > \mu$) of a limiting distribution.  In the mean field approximation, this model is essentially equivalent to the logistic Markov chain.  This chain is ergodic, i.e., it has a stationary law.  We have studied, here, the existence of the stationary regime and, in addition, have treated it analytically, for example, providing several limit theorems.


\end{document}